\documentclass[journal]{IEEEtran}
\usepackage{graphicx,times,amsmath,color,amssymb,cite,amsthm,cuted}
\usepackage{caption}
\usepackage{subcaption}
\usepackage{epstopdf}

\usepackage[a4paper,left=1in,right=1in,top=1in,bottom=1in]{geometry}

\newtheorem{theorem}{Theorem}
\newtheorem{lemma}{Lemma}

\newtheorem{remark}{Remark}
\newtheorem{assumption}{Assumption}
\newtheorem{corollary}{Corollary}

\newtheorem{problem}{Problem}
\newtheorem{proposition}{Proposition}

\def\ee{{\epsilon}}

\def\be{\begin{equation}}
\def\bea{\begin{eqnarray}}
\def\beas{\begin{eqnarray*}}
\def\eea{\end{eqnarray}}
\def\eeas{\end{eqnarray*}}

\def\bi{\begin{itemize}}
\def\ee{\end{equation}}
\def\ei{\end{itemize}}

\def\z1{z^{-1}}
\def\la{\label}

\ifCLASSINFOpdf
  \else
 \fi

\begin{document}
\title{Robust and Scalable Tracking of Radiation Sources with Cheap Binary Proximity Sensors}

\author{Henry~E.~Baidoo-Williams,~\IEEEmembership{Member,~IEEE,}
\thanks{H. E. Baidoo-Williams was with the US Army Research Laboratory, Aberdeen Proving Ground, MD and 
University at Buffalo, SUNY, NY, 14150, USA e-mail: (henrybai AT buffalo DOT edu).}
}

\maketitle

\begin{abstract}
We present a new approach to tracking of radiation sources moving on smooth trajectories which can be approximated with piece-wise linear joins or piece-wise linear parabolas. We employ the use of cheap binary proximity sensors which only indicate when a radiation source enters and leaves its sensing range. We present two separate cases. The first is considering that the trajectory can be approximated with piece-wise linear joins. We develop a novel scalable approach in terms of the number of sensors required. Robustness analysis is done with respect to uncertainties in the timing recordings by the radiation sensors. We show that in the noise free case, a minimum of three sensors will suffice to recover one piece of the linear join with probability one, even in the absence of knowledge of the speed and statistics of the radiation source. Second, we tackle a more realistic approximation of trajectories of radiation sources -- piece-wise parabolic joins -- and show that no more than six sensors are required in the noise free case to track one piece of the parabola with probability one. Next we present an upper bound on the achievable error variance in the estimation of the constant speed and the angle of elevation of linear trajectories. Finally, a comprehensive set of simulations are presented to illustrate the robustness of our approach in the presence of uncertainties.
\end{abstract}


\IEEEpeerreviewmaketitle

\section{Introduction}
\IEEEPARstart{T}{racking} and localization of radiation sources based on the received photon count at radiation sensors has garnered recent interest\cite{liu}, \cite{nem}, \cite{how}, \cite{bw}. The problem of tracking the path of a source using binary proximity sensors has been around for some time \cite{bus},\cite{kim}, \cite{shri}, \cite{aslam}, \cite{shri2}, \cite{bai}, \cite{ngu}. For tracking of radiation sensors, there is a great motivation for using binary proximity sensors. It is easy to argue that the intrinsic nature of radiation sources and their adverse effects require the surveillance of vast regions. Coupled with the fact that radiation sensors -- binary or otherwise -- are not as cheap as compared to say wireless sensors, a delicate balance has to be made to use these multiple error prone but not so cheap radiation sensors to attain reasonable errors in tracking. 

In this paper, we consider the use of cheap binary proximity sensors in tracking under some fairly general assumptions on the smoothness and continuity of source trajectory. We consider that the trajectory of the radiation source can be fairly approximated with piece-wise joins. To scenarios are considered -- piece-wise linear and piece-wise parabolic. In our formulation, we consider tracking using the exponential distributed arrival times of the photons and the total Poisson distributed count at radiations sensors. To be precise, we consider binary sensors which co-ordinate with a fusion center. The binary proximity sensors record two transition times and relay this information to the fusion center for tracking. The two transition times are the times when a radiation source enters the detection range of the sensor and when it leave it. We assume that all observations are available at the fusion center, thus no consideration is made in this paper for using distributed algorithms at the sensors to reduce overhead costs. 

The rest of the paper is organized as follows: section \ref{relw} presents a brief survey of related literature. Section \ref{skc} presents a summary of our contributions.  Section \ref{sps} presents a formal statement of the problems presented in this paper. Section \ref{sec:prelim} presents preliminary solutions and algorithms with particular application to the noise free case. Section \ref{sec:num} presents results on scalability and robustness of the least squares algorithms presented in this paper. Section \ref{sab} presents analysis on the upper bound on the achievable Cram\'{e}r-Rao lower bound on the estimation errors of the constant speed and angle of elevation of linear trajectories. Sections \ref{ssim} presents a comprehensive set of simulations to illustrate the efficacy of the algorithms developed in this paper whiles section \ref{scon} concludes.

\section{Related work} \la{relw}
Binary sensors present a rich toolset for covering large regions of interest. However, by nature, they are very noisy, and depending on the physics of the signals of interest, may compound the attainable accuracies of estimated parameters. Radiation sources,  by nature are inherently noisy even without any ``{\it conventional}" background noise. This is because the received signals are either the photon count per unit time which is poisson distributed or the inter-arrival times of the photons which are exponentially distributed. 

The existing studies on tracking of sources on piece-wise linear trajectories using binary proximity sensors considered a least squares approach \cite{bw} and Newton based approach in \cite{bw2}. In \cite{bw}, the authors showed that using four generically placed sensors, each piece of the linear joins can be estimated precisely in the no noise case. The authors did a robustness analysis, however, on the detection range of the sensors. This approach, which will be appropriate in many physical phenomena like wireless and acoustic signals, is inadequate for radiation sources. In radiation sources, the errors are in the arrival times. The approach in \cite{bw} also assumes implicitly that the source signal strength at unit distance and no shielding is known. This translates to knowledge of the detection range in the noise free case. It is worthy of note that without the knowledge of the detection range, the algorithm in \cite{bw} is not implementable. Further, the algorithm in \cite{bw} produces $2^{n-1}$ possible solutions obtained from $2^{n-1}$ least squares problems where $n$ is the number of sensors being utilized by linear trajectory. This is not scalable. In \cite{bw2}, the authors consider tracking of piece-wise linear trajectories with unknown detection range in the noise free case and show that four generically placed sensors will suffice to localize a single linear trajectory in the noise free case with probability one. The authors however use a Newton based algorithm with no guarantee of convergence unless initialization is done in the basin of attraction of the true solution. Again, the authors did a robustness analysis on the detection range. 

Other existing studies do not consider piece-wise joins, but rely on high density of sensors to guarantee that every motion is within the detection range of multiple sensors \cite{shri}, \cite{shri2}, \cite{wang}. These studies also assume implicitly that the source signal strength at unit distance and no shielding is known.

\section{Our contributions}  \la{skc}
This brings us to the key contribution of this paper:
\begin{enumerate}
\item We consider tracking of radiation sources with binary proximity sensors on piece-wise linear trajectories and make the following novel contributions:
\begin{itemize}
\item We show that contrary to existing literature, three generically placed sensors, not four, suffice to estimate a linear trajectory with probability one in the absence of noise even if we have no knowledge of the statistics of the radiation source. 
\item Our least squares algorithm for tracking on linear trajectories requires precisely 2 least squares problems as compared to $2^{n-1}$ in \cite{bw}.  
\item We consider a realistic model -- for tracking of radiation sources -- in the sense that the noise is in the arrival times and is Erlang distributed with errors scaling as a squared proportion of the arrival times at the boundaries of the unknown detection range. 
\end{itemize}
\item We extend the results from \cite{bw} and \cite{bw2} to a piece-wise parabolic trajectories with a realistic error model in the inter-arrival times corresponding to the time of shortest distance of approach to sensors. We show that no more that six sensors suffice in the noise free case to localize a linear trajectory with probability one. 
\end{enumerate}

\section{Problem Statement} \la{sps}
We now formally state the problems to be tackled in this paper. The problems are categorized under: piece-wise linear and piece-wise parabolic. 
\subsection{Tracking of a radiation source on a piece-wise linear trajectory}
Consider a mobile radiation source in $\mathbb{R}^2$ on a piece-wise linear trajectory with the $i^{th}$ piece of the linear trajectory defined as:
\bea \la{ltx}
\begin{split}
x_i(t) &= x_{oi}^* +s_{xi}^* (t-t_{oi}) \sin{\theta_i^*} \\
y_i(t) &= y_{oi}^* +s_{yi}^* (t-t_{oi}) \cos{\theta_i^*}
\end{split}
\eea
where $\{x_{oi}^*, y_{oi}^*\}$ is the location of the radiation source at time instant $t =t_{oi}$. $\{s_{xi}^*, s_{yi}^* \}$ are constant speeds with respect to the $x$ and $y$ axis respectively in the cartesian plane. ${\theta_i}$ is the elevation angle with respect to the $x$ axis of the cartesian plane. 

Consider $n$ binary proximity sensors in $\mathbb{R}^2$ under the following assumption:
\begin{assumption} \la{pso}
The location of each sensor is a point source $\{x_j, y_j\}, j \in \{1, \cdots, m\}$. Further, the time instantaneous location of the radiation source $\{x_i(t), y_i(t)\}, i \in \{1, \cdots, n\}$ is also a point source.
\end{assumption}
Under assumption \ref{pso}, the time instantaneous distance of the $j^{th}$ radiation sensor, $, \in \{1, \cdots, n\}$,  on the $i^{th}$ linear trajectory, $i \in \{1, \cdots, m\}$ is defined as:
\bea \la{dist}
d_{ij}(t)=\sqrt{(x_i(t) - x_j)^2+(y_i(t) - y_j)^2}
\eea
The binary proximity radiation sensors measure the number of photons received at the $j^{th}$ sensor, $j \in \{1, \cdots, n\}$,  and output a time stamp or not depending on whether the aggregate number of received photons within a time window is greater or less than an arbitrary threshold which is set a priori.  Time transitions corresponding to a sensors moving form outside the detection range to inside the detection range of sensors are recorded and transmitted to a fusion center. Precisely, the $j^{th}$ sensor's measurements, $j \in \{1, \cdots, n\}$, which are transmitted to the fusion center whiles the radiation source is on the $i^{th}$ piece of the linear trajectory, $i \in \{1, \cdots, m\}$,  in the noise free case are:
\bea \la{tinit}
t_{ij1}^* = \underset{t}{\operatorname{argmin}} \int\limits_{t-\frac{1}{2}}^{t+\frac{1}{2}} \frac{e^{-\alpha_{s} d_{ij}^*(\tau)}}{d_{ij}^{*2}(\tau)} d\tau \ge \frac{\lambda_T -\lambda_b}{\lambda_s^* } 
\eea
\bea \la{tfinal}
t_{ij2}^* = \underset{t}{\operatorname{argmax}} \int\limits_{t-\frac{1}{2}}^{t+\frac{1}{2}} \frac{e^{-\alpha_{s} d_{ij}^*(\tau)}}{d_{ij}^{*2}(\tau)} d\tau \ge \frac{\lambda_T -\lambda_b}{\lambda_s^* } 
\eea
From (\ref{tinit}) -- (\ref{tfinal}), the time of shortest distance of approach to each sensor $j \in \{1, \cdots, n\}$ corresponding from a particular linear trajectory $i \in \{1, \cdots, m\}$ becomes:
\bea \la{tstar}
t_{ij}^* =\frac{t_{ij1}^*+t_{ij2}^*}{2}
\eea
Here, $\lambda_s^*$ is the photon count at unit distance and zero shielding per second. $\alpha_{s}$ is the shielding coefficient which depends on type and geometric shape of the radiation material. $\lambda_T > \lambda_b$ implicitly defines the expected maximum radius beyond which the sensor is ``{\it off}", and ``{\it on}" otherwise. $\lambda_b$ is normally occurring radiation background noise. In reality, both $\lambda_s^*$ and $\lambda_b$ are poisson random variables.

Suppose the following assumptions hold:
\begin{assumption} \la{ass:one}
The radiation source intensity $\lambda_s$, is unknown.
\end{assumption} 
\begin{assumption} \la{ass:two}
The location of the sensors are all distinct,  independent and identical distributed and zero mean along both the $x$ and $y$ axis in the cartesian plane.
\end{assumption} 
\begin{assumption} \la{ass:three}
The radiation source moves with a constant speed for each piece of the trajectory. Thus $\{ s_{xi}^* , s_{yi}^* \} = s_i^*~ \forall~ i, i \in \{1, \cdots, m\}$. 
\end{assumption}
\begin{assumption} \la{ass:four}
The received radiation photon count at each sensor with respect to any single piece of the trajectory, under noise free case, is unimodal. 
\end{assumption}

It is clear that the expected instantaneous received signal at sensor $j$ with respect to the $i^{th}$ trajectory is given as:
\bea \la{nf}
z_{ij}(t) = \frac{\lambda_s^* e^{-\alpha_s d_{ij}(t)}}{d_{ij}^{2}(t)} + \lambda_b
\eea
The expected value of the peak with respect to time of (\ref{nf}) obeys: 
\begin{align*}
\begin{split}
&\dot{z}_{ij}(t) = \frac{\lambda_s^* e^{-\alpha_s d_{ij}(t)}}{d_{ij}^{2}(t)}\left( \alpha_s +\frac{2}{d_{ij}(t)}\right) \frac{1}{d_{ij}(t)} \dot{d}_{ij}(t) \Rightarrow\\
&\dot{d}_{ij}(t_{ij}^*) =0 \Rightarrow\\
&0 =\frac{(x_i(t_{ij}^*)-x_{j})s_i^*\sin{\theta_i^*}+(y_i(t_{ij}^*)-y_{j})s_i^*\cos{\theta_i^*}}{d_{ij}^*(t_{ij}^*)}
\end{split}
\end{align*}
leading to:
\bea \la{tstar0}
 t_{ij}^* -t_{oi}= \frac{\sin{\theta_i^*}}{s_i^*} \left(x_j -x_{oi}^*\right) +  \frac{\cos{\theta_i^*}}{s_i^*} \left(y_j -y_{oi}^*\right)
\eea
This brings us to the first  problem statement:
\begin{problem} \la{prob:one}
Under (\ref{ltx}), suppose assumptions \ref{pso}, \ref{ass:one}, \ref{ass:two}, \ref{ass:three} and \ref{ass:four} hold.  Given:
\begin{enumerate}
\item the sensor readings $\{t_{ij1}^*,t_{ij2}^*\}$ , $i \in \{1, \cdots, m\}$, $j \in \{1, \cdots, n\}$ 
\item the sensor locations $\{x_j, y_j \}~ \forall j \in \{1, \cdots, n\}$
\end{enumerate}
\begin{itemize}
\item can we estimate $\{x_{oi}^*, y_{oi}^*,  \theta_i^*, s_{i}^*\}$? 
\item is the estimate unique? 
\item what is the minimum number of sensor readings required for estimation of a single linear piece?
\item suppose there is noise in the recorded times, can we still estimate the line robustly?
\end{itemize}
\end{problem} 
\subsection{Tracking of radiation source on a piece-wise parabolic trajectories}
Now consider a  radiation source in $\mathbb{R}^2$ whose $i^{th}$ piece of a piece-wise parabolic trajectory obeys:
\bea \la{ptx}
x_i(t) &= x_{oi}^* +\alpha_{i}^* (t-t_{oi}) \\
y_i(t) &= y_{oi}^* +\beta_i^* (t-t_{oi}) +\frac{1}{2} \gamma_i^* (t-t_{oi})^2
\eea
where $\{x_{oi}^*, y_{oi}^*\}, i \in \{1,\cdots, m\}$ is the location of the radiation source at time instant $t =t_{oi}, i \in \{1,\cdots, m\}$. Consequently,  the $j^{th}$ sensor's measurements corresponding to the time of shortest approach, $j \in \{1, \cdots, n\}$,  whiles on the $i^{th}$ piece of the parabolic trajectory, $i \in \{1, \cdots, m\}$, is as defined in (\ref{tstar}). The time of shortest approach of the $i^{th}$ trajectory from the $j^{th}$ sensor following a similar derivation leading to (\ref{tstar0}) is:
\begin{align} \la{tstar1}
\frac{1}{2 }& \gamma_i^*  (t_{ij}^*-t_{oi})^{3}  +\frac{3}{2}  \beta_i^* (t_{ij}^*-t_{oi})^{2} +\nonumber \\
&\left (\frac{\alpha_i^{*2}+\beta_i^{*2}}{ \gamma_i^*}+y_{oi}^*\right) (t_{ij}^{*}-t_{oi}) +\frac{\alpha_i^*}{ \gamma_i^*} (x_{oi}^* -x_j)  +\nonumber \\
&\frac{\beta_i^*}{ \gamma_i^*} (y_{oi}^* -y_j) = y_j  (t_{ij}^{*} -t_{oi})
\end{align}
This brings us to the second  problem statement:
\begin{problem} \la{prob:one}
Under (\ref{ptx}), suppose assumptions \ref{pso}, \ref{ass:one}, \ref{ass:two}, \ref{ass:three} and \ref{ass:four} hold.  Given:
\begin{enumerate}
\item the sensor readings of the time of shortest approach $t_{ij}^*$ , $i \in \{1, \cdots, m\}$, $j \in \{1, \cdots, n\}$ 
\item the sensor locations $\{x_j, y_j \}~ \forall j \in \{1, \cdots, n\}$
\end{enumerate}
\begin{itemize}
\item can we estimate $\{x_{oi}^*, y_{oi}^*,  \alpha_i^*, \beta_{i}^*, \gamma_i^*\}$? 
\item is the estimate unique? 
\item what is the minimum number of sensor readings required for estimation of a single parabolic piece?
\item suppose there is noise in the recorded times, can we still estimate the parabola robustly?
\end{itemize}
\end{problem} 

\section{Preliminaries}\la{sec:prelim}
We now consider the two problems described in section \ref{sps}.
\subsection{Tracking of piece-wise linear trajectories under no noise conditions}

Suppose that a radiation source is on a line as defined in (\ref{ltx}) and suppose the radiation source triggers $n$ radiation sensors such that their time of shortest distance of approach $t_{ij}^*-t_{oi}$, $i \in \{1, \cdots, m\}, ~j \in \{1,\cdots,n\}$, is recorded and transmitted to the fusion center. Also, suppose the locations of the radiation sensors are known. The question being asked is whether we can uniquely determine the parameters of (\ref{ltx}) from the given information. To begin, we have the following theorem:
\begin{theorem} \la{th:one}
Suppose a radiation source on a linear trajectory, as defined in (\ref{ltx}), triggers $n = 3$ sensors such that:
\begin{itemize}
\item the times stamps , $\{t_{ij1}^*, t_{ij2}^*\}, i \in \{1, \cdots, m\}, ~j \in \{1,2,3\}$ are available.
\item the locations of the centers of the $n = 3$  radiation sensors are independent and identically distributed in $\mathbb{R}^2$.
\end{itemize}
 Then with probability being equal to one, the line (\ref{ltx}) is uniquely determined based on the $n = 3$ radiation sensor centers centers $(x_j, y_j), j \in \{1,2,3\}$ and the $\{t_{ij1}^*, t_{ij2}^*\}, i \in \{1, \cdots, m\}, ~j \in \{1,2,3\}$,  information.
\end{theorem}
\begin{proof}
Without loss of generality, suppose $x_1 =y_1 = y_2 =0$. Notice that this can be achieved by rotation and translation of the entire space without altering any of the measured signal because distance measurements are invariant under translation and rotation. Also note that the parameters $\{\theta_i^*, x_{oi}^*, y_{oi}^*\}$ are arbitrarily rotated with respect to $\{x_1,y_1,y_2\}$. $\{x_{oi}^*, y_{oi}^*\}$ are also translated. We however retain all notations to preserve clarity. Consequently, from (\ref{tstar}):
\bea \la{sol:one}
t_{i1}^* =-\frac{\sin{\theta_i^*}}{s_i^*}x_{oi}^* -\frac{\cos{\theta_i^*}}{s_i^*}y_{oi}^*+t_{oi}
\eea
\bea \la{sol:two}
t_{i2}^* =\frac{\sin{\theta_i^*}}{s_i^*}x_2 -t_{i1}^*
\eea
\bea \la{sol:three}
t_{i3}^* =\frac{\sin{\theta_i^*}}{s_i^*}x_3 + \frac{\cos{\theta_i^*}}{s_i^*}y_3-t_{i1}^*
\eea
Here, (\ref{sol:one}) has already been utilized to form (\ref{sol:two}) and (\ref{sol:three}). From (\ref{sol:two}) and (\ref{sol:three}):
\bea \la{sol:four}
s_i^* = \frac{\sin{\theta_i^*}}{t_{i1}^*+t_{i2}^*} x_2 \Rightarrow
\eea
\bea \la{sol:five}
t_{i3}^* &= (t_{i1}^*+t_{i2}^*) \left (\frac{x_3}{x_2} +\frac{1}{\tan{\theta_i^*}} \frac{y_3}{x_2}\right) -t_{i1}^* 
\eea
\bea \la{sol:six}
\Rightarrow \tan{\theta_i^*}  = \frac{y_3}{\frac{t_{i1}^*+t_{i3}^*}{t_{i1}^*+t_{i2}^*}x_2 - x_3}
\eea
Notice that since $\{x_2,x_3,y_3\} \neq 0$, there is no degenerate solution in (\ref{sol:six}). Consequently $\{\theta_i^*, s_i^*\}$ are uniquely determined as:
\bea \la{sol:seven}
\hat{\theta}_i^* =\tan^{-1}{\left( \frac{y_3}{\frac{t_{i1}^*+t_{i3}^*}{t_{i1}^*+t_{i2}^*}x_2 - x_3}\right)}
\eea
\bea \la{sol:eight}
\hat{s}_i^* =\frac{x_2}{t_{i1}^*+t_{t_i2}^*}
 \sin{
 \left(
 \tan^{-1}
 {\left( \frac{y_3}{\frac{t_{i1}^*+t_{i3}^*}{t_{i1}^*+t_{i2}^*}x_2 - x_3}\right)}
 \right)
 } 
\eea
Now, we move on to determe $\{x_{oi}, y_{oi}\}$.  Notice that (\ref{ltx}) can be rewritten devoid of explicit dependence on time as:
\bea \la{nltx}
(y-y_{oi})\sin{\theta_i^*} - (x- x_{oi} ) \cos{\theta_i^*}=0
\eea
The distance of shortest approach, can easily be shown, for each sensor to be: 
$$
(y_j-y_{oi})\sin{\theta_i^*} - (x_j- x_{oi} ) \cos{\theta_i^*}=0
$$
Since the maximum distant from which any of the sensors is triggered is the same, we can use any pair of sensor, to arrive at the following equation:
\bea \la{temp1}
\begin{split}
&((y_j-y_{oi})\sin{\theta_i^*} - (x_j- x_{oi} ) \cos{\theta_i^*})^2 \\
&+\frac{s_i^{*2}}{4} \left( t_{ij2}^*-t_{ij1}^*\right)^2 = \frac{s_i^{*2}}{4} \left( t_{ik2}^*-t_{ik1}^*\right)^2 \\
&((y_k-y_{oi})\sin{\theta_i^*} - (x_k- x_{oi} ) \cos{\theta_i^*})^2 
\end{split}
\eea
Further rearrangement of (\ref{temp1}) will lead to:
\bea
\begin{split} \nonumber
&x_{oi}\left(2 (x_k -x_j)\cos^2{\theta_i^*} -2 (y_k -y_j)\sin{\theta_i^*}\cos{\theta_i^*} \right)+\\
&y_{oi}\left(2 (y_k -y_j)\sin^2{\theta_i^*} -2 (x_k -x_j)\sin{\theta_i^*}\cos{\theta_i^*} \right)\\
&=\frac{s_i^{*2}}{4} \left((t_{ik2}-t_{ik1})^2-(t_{ij2}-t_{ij1})^2 \right) +\\
&(y_k \sin{\theta_i^*}-x_k \cos{\theta_i^*})^2-(y_j \sin{\theta_i^*}-x_j \cos{\theta_i^*})^2
\end{split}
\eea
Recall that by assumption, $\{x_1, y_1, y_2\}=0$. Hence, we can reduce (\ref{temp1}) further to:
\bea \la{sol1n}
\begin{split}
&2 x_2\cos^2{\theta_i^*} x_{oi}-2 x_2 \sin{\theta_i^*}\cos{\theta_i^*} y_{oi} =\\
&\frac{s_i^{*2}}{4} \left((t_{ik2}-t_{ik1})^2-(t_{ij2}-t_{ij1})^2 \right) + x_2^2 \cos^2{\theta_i^*}
\end{split}
\eea
Now, solving (\ref{sol1n})  concurrently with (\ref{sol:one}), a degenerate solution will arise if and only if:
$$
x_2 \cos{\theta_i^*} = 0
$$
Since $x_2 \neq 0$, we will have a degenerate solution if and only if $\theta_i^* \in \{\frac{\pi}{2}, \frac{3\pi}{2} \}$. Notice that we have rotated and translated the entire space to have a situation where $\{x_1, y_1, y_2\}=0$. Since $\{x_1, y_1, y_2\}=0$ are i.i.d. distributed, the probability that the angle of elevation after rotation about $\{x_1, y_1\}$ such that $y_2 =0$ is in $\{\frac{\pi}{2}, \frac{3\pi}{2} \}$ is zero. This concludes the proof that $\{\theta_i^*, s_i^*, x_{oi}^*, y_{oi}^*\}$ and by extension the line (\ref{ltx}) is uniquely determined with probability being equal to one.
\end{proof}
\subsection{Tracking of piece-wise parabolic trajectories under no noise conditions}

Suppose that a radiation source is on a parabola as defined in (\ref{ptx}) under no noise conditions. Suppose the radiation source triggers $n$ radiation sensors such that their time of shortest distance of approach $t_{ij}^*$, $i \in \{1,\dots,m\}, ~j \in \{1,\cdots,n\}$, is recorded and transmitted to a fusion center for processing. Also, suppose the locations of the binary radiation measurement sensors are known. The question being asked is whether we can uniquely determined the parameters of (\ref{ptx}) from the given information.

\begin{theorem} \la{th:onea}
Suppose a radiation source on a parabolic trajectory, as defined in (\ref{ptx}), triggers $n=6$ sensors such that:
\begin{itemize}
\item the times corresponding to the shortest distance of approach, $t_{ij}^*, j \in \{1,\cdots,6\}$, are available.
\item the locations of the centers of the radiation sensors are independent and identically distributed in $\mathbb{R}^2$.
\end{itemize}
 Then with probability being equal to one, the parabola (\ref{ptx}) is uniquely determined based on the $n=6$ measurement sensor point locations $(x_j, y_j), j \in \{1,\cdots,6\}$ and the times of shortest distance of approach, $t_{ij}^*, ~j \in  \{1,\cdots,6\}$,  information.
\end{theorem}

\begin{proof}
Without loss of generality, suppose $x_1 =y_1 = y_2 =0$. Again, we retain all notations, to preserve clarity. Also consider $t_{oi} =0$. Consequently, from (\ref{tstar1}):
\bea \la{sol:onea}
&\frac{1}{2} \gamma_i^*  t_{ij}^{*3}  +\frac{3}{2}  \beta_i^*  t_{ij}^{*2} +\left (\frac{\alpha_i^{*2}+\beta_i^{*2}}{ \gamma_i^*}+y_{oi}^*\right) t_{i1}^{*} \nonumber \\
&+\frac{\alpha_i^*}{ \gamma_i^*} (x_{oi}^* -x_j)  +\frac{\beta_i^*}{ \gamma_i^*} (y_{oi}^* -y_j)  = y_j t_{ij}^*
\eea
$\forall j \in \{1,\cdots, n\}.$
In matrix form, (\ref{sol:onea}) becomes:
{\scriptsize $$\underbrace{\begin{bmatrix}
       1 & t_{i1}^{3*}& t_{i1}^{2*} & t_{i1}^{*}   & 0 & 0       \\[0.3em]
       1 & t_{i2}^{3*}& t_{i2}^{2*} & t_{i2}^{*}   & x_2 & 0      \\[0.3em]
       1 & t_{i3}^{3*}& t_{i3}^{2*} & t_{i3}^{*}   & x_3 & y_3      \\[0.3em]
       1 & t_{i4}^{3*}& t_{i4}^{2*} & t_{i4}^{*}   & x_4 & y_4      \\[0.3em]
       1 & t_{i5}^{3*}& t_{i5}^{2*} & t_{i5}^{*}   & x_5 & y_5      \\[0.3em]
       1 & t_{i6}^{3*}& t_{i6}^{2*} & t_{i6}^{*}   & x_6 & y_6      \\[0.3em]
     \end{bmatrix}}_{\mathbb{A}_o}
\underbrace{\begin{bmatrix}
      \frac{\alpha_i^*}{\gamma_i^*} x_{oi}^* + \frac{\beta_i^*}{\gamma_i^*} y_{oi}^*        \\[0.2em]
       \frac{1}{2} \gamma_i^*        \\[0.2em]
       \frac{3}{2} \beta_i^*         \\[0.2em]
      \frac{\alpha_i^{*2}+\beta_i^{*2}}{\gamma_i^*} + y_{oi}^*        \\[0.2em]
        -\frac{\alpha_i^*}{\gamma_i^*}      \\[0.2em]
        -\frac{\beta_i^*}{\gamma_i^*}       \\[0.2em]
     \end{bmatrix}}_{\mathbb{X}_o}
= \begin{bmatrix}
       0        \\[0.2em]
       0        \\[0.2em]
       y_3 t_{i3}^*       \\[0.2em]
     y_4 t_{i4}^*        \\[0.2em]
       y_3 t_{i3}^*      \\[0.2em]
       y_5 t_{i5}^*           \\[0.2em]
       y_6 t_{i6}^*           \\[0.2em]
     \end{bmatrix}$$}
With some algebraic manipulations, it can be shown that:

{\scriptsize 
$$\Gamma
\left[ 
\left(\frac{x_5}{\kappa_{5}}-\frac{x_4}{\kappa_{4}}\right)
\left(\frac{y_6}{\kappa_{6}}-\frac{y_4}{\kappa_{4}}\right)-
\left(\frac{x_6}{\kappa_{6}}-\frac{x_4}{\kappa_{4}}\right)\left(\frac{y_5}{\kappa_{5}}-\frac{y_4}{\kappa_{4}}\right) 
\right]$$
}
where $\kappa_{m} = \prod\limits_{k=1}^3 (t_{im}-t_{ik})$ and $\Gamma =\prod\limits_{1 \le j \le  3  \atop 4 \leq k \leq 6 } (t_{ik}^* -t_{ij}^*) $. This means for a degenerate solution:
\bea \la{sol:twoa1}
\frac{y_6 x_5}{\kappa_{(5,6)}}-\frac{y_6 x_4}{\kappa_{(4,6)}}- \frac{x_6 y_5}{\kappa_{(5,6)}}-\frac{x_6 y_4}{\kappa_{(4,6)}} = \frac{x_5 y_4 -x_4 y_5}{\kappa_{(5,6)}}
\eea
where $\kappa_{(m,n)} = \prod\limits_{k=1 \atop j \in \{m,n\}}^3 (t_{ij}-t_{ik})$. 
Now, suppose we know $\{x_j, y_j\}, j \in \{1,\cdots,5\}$, $\{x_6, y_6\}$ will have to lie on a one-dimensional hyperplane in a two-dimensional probability space. Since $\{x_6, y_6\}$ are independent and identically distributed, with probability 0, (\ref{sol:twoa1}) will not be realized meaning with probability being one, the parabolic trajectory is uniquely determined from (\ref{sol:twoa1}). Notice that from $\mathbb{X}_0$, all the parameters of the parabola are uniquely determined. 
\end{proof}

\section{Robust and scalable numerical solution to tracking of radiation sources on piece-wise linear and parabolic trajectories} \label{sec:num}
Up till this point, we have been working with the expectations of the received signals at the radiation sensors. We will now consider the practical scenario with noise in the observed arrival times of the photons. 

\subsection{Scalable tracking of radiation source on piece-wise linear trajectories} \label{subsec:num1}

From Theorem \ref{th:one}, we have shown that given three generically placed radiation sensors such that they are non-coincident, we can track a linear trajectory precisely with probability being 1 if the sensor measurements are noise free. In this section, we will present a robust and scalable algorithm  in terms of noisy measurement and number of sensors being used respectively. Recall that, other results on tracking with binary sensors on piece-wise linear trajectories required $2^{n-1}$ least squares problems and comparisons to come up with the true solution\cite{bai} . In our approach, we will show that only two least squares problems are required for all $n \ge 3$. Observe that (\ref{tstar}) can be written in matrix form as:

{\scriptsize
\bea \la{sol:threea}
\underbrace{\begin{bmatrix}
       x_1 & y_1 & 1     \\[0.3em]
       x_2 & y_2 & 1     \\[0.3em]
       \vdots & \vdots & \vdots \\[0.3em]
       x_n & y_n & 1     \\[0.3em]
     \end{bmatrix}}_{\mathbb{A}_1} 
\underbrace{\begin{bmatrix}
       \frac{\sin{\theta_i^*}}{s_i^*}  \\[0.2em]\\
       \frac{\cos{\theta_i^*}}{s_i^*}   \\[0.2em]
        -\frac{1}{s_i^*}(\sin{\theta_i^*} x_{oi}^* + \cos{\theta_i^*} y_{oi}^*)  \\[0.2em]
     \end{bmatrix} }_{\mathbb{X}_1}
= \underbrace{\begin{bmatrix}
      t_{i1}^*-t_{oi}        \\[0.2em]
      t_{i2}^*-t_{oi}        \\[0.2em]
       \vdots       \\[0.2em]
       t_{in}^*-t_{oi}       \\[0.2em]
     \end{bmatrix}}_{\mathbb{Y}_1}
\eea
}

The least squares solution can be used to form:
\bea \la{sol:threeb}
\mathbb{X}_1 =(\mathbb{A}_1^T \mathbb{A}_1)^{-1} \mathbb{A}_1^T \mathbb{Y}_1
\eea
and note that:
\bea \la{sol:threec}
\hat{s}_i^* =\sqrt{\mathbb{X}_1^2 (1) +\mathbb{X}_1^2 (2)} 
\eea
\bea \la{sol:threed}
\hat{\theta}_i^* =\arctan{\left(\frac{\mathbb{X}_1 (1)}{\mathbb{X}_1 (2)}\right)}
\eea
 From (\ref{temp1}), we can formulate:
\bea \la{sol:foura}
\begin{split}
\underbrace{\begin{bmatrix}
       \alpha_{2 (1,2)} & \beta_{2(1,2)}     \\[0.3em]
      \vdots & \vdots    \\[0.3em]
      \alpha_{2 (1,n)} & \beta_{2(1,n)}     \\[0.3em]
      -\frac{\sin{\hat{\theta}_i^*}}{\hat{s}_i^*} & -\frac{\cos{\hat{\theta}_i^*}}{\hat{s}_i^*}     \\[0.3em]
           \end{bmatrix}}_{\mathbb{A}_2} 
\underbrace{\begin{bmatrix}
      x_{oi}  \\[0.2em]\\
      y_{oi}   \\[0.2em]
     \end{bmatrix} }_{\mathbb{X}_2}
= 
\underbrace{\begin{bmatrix}
    \gamma_{2(1,2)}     \\[0.2em]
       \vdots       \\[0.2em]
       \gamma_{2(1,n)}     \\[0.2em]
	\mathbb{X}_1 (3) \\[0.2em]
     \end{bmatrix}}_{\mathbb{Y}_2}
\end{split}
\eea
where 
\bea \nonumber
\begin{split}
&\alpha_{2(k,j)} = 2 (x_k -x_j)\cos^2{\hat{\theta}_i^*} -2 (y_k -y_j)\sin{\hat{\theta}_i^*}\cos{\hat{\theta}_i^*} \\
&\beta_{2(k,j)}  = 2 (y_k -y_j)\sin^2{\hat{\theta}_i^*} -2 (x_k -x_j)\sin{\hat{\theta}_i^*}\cos{\hat{\theta}_i^*}\\
&\gamma_{2(k,j)} = \frac{\hat{s}_i^{*2}}{4} \left((t_{ik2}-t_{ik1})^2-(t_{ij2}-t_{ij1})^2 \right) +\\
& (y_k \sin{\hat{\theta}_i^*}-x_k \cos{\hat{\theta}_i^*})^2-(y_j \sin{\hat{\theta}_i^*}-x_j \cos{\hat{\theta}_i^*})^2
\end{split}
\eea
The least squares solution then becomes:
\bea \la{sol:fourb}
\mathbb{X}_2 =(\mathbb{A}_2^T \mathbb{A}_2)^{-1} \mathbb{A}_2^T \mathbb{Y}_2.
\eea
and note that the estimates:
$\hat{x}_{oi} =\mathbb{X}_2 (1)$, and 
$\hat{y}_{oi} =\mathbb{X}_2 (2)$.
Notice that we have used the least squares exactly twice. 
\subsection{Noise in recorded times for radiation source on piece-wise linear trajectories}\la{subsec:num2}

We now consider the robustness with reference to the presence of noise. Recall that the arrival time  of each photon is at each sensor is exponentially distributed. Now suppose the following proposition holds in the noisy case:
\begin{proposition} \label{noise1}
The recorded time $t_{ijk}-t_{oi}, k \in \{1,2\}$ corresponding to the earliest and latest sensing by the  $j^{th}$ sensor and the $i^{th}$ piece of the linear trajectory obeys:
$t_{ijk}-t_{oi} \sim {Erlang}\left(\lambda_T, \frac{\lambda_T}{t_{ijk}^*-t_{oi}}\right)$.
\end{proposition}
\begin{proof}
Suppose that the $\lambda_T$ arrival times in the interval corresponding to $t_{ijk}$ are denoted $ \tau_{ijk}^1 < \tau_{ijk}^2< \cdots < \tau_{ijk}^{\lambda_T}$, $\tau_{ijk}^{\lambda_T}-\tau_{ijk}^{1} \leq 1$, where the superscripts are just for labelling purposes. Suppose $\frac{1}{\lambda_T} \sum\limits_{q=1}^{\lambda_T} \tau_{ijk}^q -t_{oi}\approx t_{ijk}^*-t_{oi}$. Notice that the $\tau_{ijk}^q$ poisson arrival times are exponentially distributed with approximate rate parameter $\lambda_T$. Subsequently, it is fairly easy to show that $t_{ijk} -t_{oi}  \sim  Erlang\left(\lambda_T,\frac{\lambda_T}{t_{ijk}^*-t_{oi}}\right)$.
\end{proof}
\begin{corollary}
The recorded time of shortest approach, $t_{ij}^*-t_{oi}$ obeys:
$t_{ij}^*-t_{oi} \sim {Erlang}\left(2\lambda_T, \frac{2 \lambda_T}{t_{ij}^*-t_{oi}}\right)$.
\end{corollary}
Consequent to proposition \ref{noise1}, the equations for estimation of the linear trajectory has to be updated. From (\ref{sol:threea}) - (\ref{sol:threeb}) the estimation error in $\{s_i^*, \theta_i^* \}$ can be characterized as:
\bea \la{sol:fivea}
(\mathbb{X}_1^* - \mathbb{X}_1) =(\mathbb{A}_1^T \mathbb{A}_1)^{-1} \mathbb{A}_1^T  \underbrace{\begin{bmatrix}
      t_{i1}^*-t_{i1}      \\[0.2em]
       \vdots       \\[0.2em]
       t_{in}^*-t_{ik}      \\[0.2em]
     \end{bmatrix}}_{\Delta \mathbb{Y}_1}
\eea
From (\ref{sol:fivea}), the following lemma arises:
\begin{lemma} \label{lem:fivea}
Under Assumption \ref{noise1} and (\ref{sol:fivea})
$$\lim_{k\to\infty} (\mathbb{X}_1^* - \mathbb{X}_1)  = \Delta \mathbb{X}_1\rightarrow [0 ~  0 ~  0]^T.$$
\end{lemma}
\begin{proof}
\bea \nonumber \lim_{k\to\infty}  \Delta \mathbb{X}_1\rightarrow {\text diag} \left \{ \frac{1}{\sigma_x^2}, \frac{1}{\sigma_y^2}, 1 \right\} \frac{1}{k}\mathbb{A}_1^T  \underbrace{\begin{bmatrix}
      t_{i1}^*-t_{i1}      \\[0.2em]
       \vdots       \\[0.2em]
       t_{in}^*-t_{ik}      \\[0.2em]
     \end{bmatrix}}_{\Delta \mathbb{Y}_1}
\eea
Since the $\{x_j, y_j\}$ elements of $\mathbb{A}_1^T $ are i.i.d. and zero mean, the result follows.
\end{proof}
\begin{corollary} \label{cor:fivea}
Under Assumption \ref{noise1} and (\ref{sol:fivea})
$$\lim_{k\to\infty} [s_i^* -\hat{s}_i, \theta_i^*-\hat{\theta}_i ]^T \rightarrow [0 ~  0 ]^T.$$
\end{corollary}
We now need to show that the estimates of $\{x_{oi}, y_{oi}\}$ also goes to zero in the limit. From (\ref{sol:foura})  the estimation error in $\{x_{oi}^*, y_{oi}^* \}$ can be characterized as:

\bea \la{sol:fiveb}
(\mathbb{X}_2^* - \mathbb{X}_2) =(\mathbb{A}_2^T \mathbb{A}_2)^{-1} \mathbb{A}_2^T  \underbrace{\begin{bmatrix}
     \Delta y_{1,2}     \\[0.2em]
       \vdots       \\[0.2em]
      \Delta y_{1,n}     \\[0.2em]
       \mathbb{X^*}_1 (3) -  \mathbb{\hat{X}}_1 (3) \\[0.2em]
     \end{bmatrix}}_{\Delta \mathbb{Y}_2}
\eea
where 
\bea 
\begin{split}
\Delta y_{1,j} = & \frac{\hat{s}_i^{*2}}{4} \left((t_{i12}^*-t_{i11}^*)^2-(t_{ij2}^*-t_{ij1}^*)^2 \right)-\\
&\frac{\hat{s}_i^{*2}}{4} \left((t_{i12}-t_{i11})^2-(t_{ij2}-t_{ij1})^2 \right)
\end{split}
\eea
Suppose the followed ordered statistics hold:
\bea \nonumber
\begin{split}
 &t_{i22}-t_{i21} \leq \cdots \leq t_{i,\lfloor{\frac{n}{2}}\rfloor-1,2}^*-t_{i\lfloor{\frac{n}{2}}\rfloor-1,1}  \leq t_{i12}^*-t_{i11} \\
  & \leq t_{i,\lfloor{\frac{n}{2}}\rfloor+1,2}^*-t_{i,\lfloor{\frac{n}{2}}\rfloor+1,1} \leq \cdots \leq t_{in2}-t_{in1}
\end{split}
\eea
From (\ref{sol:fiveb}), we have the following lemma:

\begin{lemma} \label{lem:fiveb}
Under Assumption \ref{noise1} and (\ref{sol:fiveb})
$$\lim_{k\to\infty} (\mathbb{X}_2^* - \mathbb{X}_2)  = \rightarrow [0 ~  0]^T.$$
\end{lemma}
\begin{proof}
Observe that the entries of $\Delta \mathbb{Y}_2$, with the exception of the last row, are zero mean Laplacian random variables. The last entry is also zero mean distributed because it is an algebraic sum of $x_j$ and $y_j$ components. Thus: $$\lim_{k\to\infty}  \mathbb{A}_2^T  \Delta \mathbb{Y}_2 \rightarrow 0.$$ The result therefore follows from this observation.
\end{proof}
\begin{corollary} \label{cor:fiveb}
Under Assumption \ref{noise1} and (\ref{sol:fiveb})
$$\lim_{k\to\infty} [x_{oi}^* -\hat{x}_{oi}, y_{oi}^*-\hat{y}_{oi} ]^T \rightarrow [0 ~  0 ]^T.$$
\end{corollary}

\subsection{Scalable tracking of radiation source on piece-wise parabolic trajectories} \label{subsec:num3}

From Theorem \ref{th:onea}, we have shown that given six generically placed radiation sensors such that they are non-collinear, we can track a parabolic trajectory precisely with probability being 1, if we are given the exact time of shortest distance of approach. In this section, we will present a robust and scalable algorithm  in terms of noise and number of sensors being used. Recall from (\ref{sol:onea}):

\scriptsize
\bea \nonumber
\underbrace{\begin{bmatrix}
       1 & (t_{i1}^*-t_{oi})^{3}& (t_{i1}^*-t_{oi})^{2} & t_{i1}^{*}-t_{oi}   & x_1 & y_1       \\[0.3em]
             \vdots & \vdots & \vdots& \vdots  & \vdots & \vdots      \\[0.3em]
       1 & (t_{in}^*\vdots)^{3}& (t_{in}^*\vdots)^{2} & t_{in}^{*}\vdots   & x_n & y_n      \\[0.3em]
     \end{bmatrix}}_{\mathbb{A}_3} 
\mathbb{X}_3 
=\mathbb{Y}_3
\eea
\normalsize

The least squares solution is:
\bea \la{sol:sixb}
\mathbb{X}_3 =(\mathbb{A}_3^T \mathbb{A}_3)^{-1} \mathbb{A}_3^T \mathbb{Y}_3
\eea
and note that:

{\scriptsize
$$\hat{\gamma}_i^* =2 \mathbb{X}_3 (2)$$ 
$$\hat{\alpha}_i^* = - 2 \mathbb{X}_3 (2)\mathbb{X}_3 (5) $$
$$\hat{\beta}_i^* = - 2 \mathbb{X}_3 (2)\mathbb{X}_3 (6) $$
$$\hat{y}_{oi}^* = \mathbb{X}_3 (4) - 2 \mathbb{X}_3 (2)\left( \mathbb{X}_3^2 (5) + \mathbb{X}_3^2 (6)\right)$$
$$\hat{x}_{oi}^* = -\frac{\mathbb{X}_3(1)}{ \mathbb{X}_3 (5)} -\left( \mathbb{X}_3 (4) - 2 \mathbb{X}_3 (2)\left( \mathbb{X}_3^2 (5) + \mathbb{X}_3^2 (6)\right)\right)\frac{\mathbb{X}_3 (6)}{\mathbb{X}_3 (5)}.$$
}

The least squares solution easily scales with number of measurement sensors.
\subsection{Noise in recorded times for radiation source on piece-wise parabolic trajectories}\la{subsec:num4}

We now consider the presence of noise in observed time stamps. The following proposition follows from proposition \ref{noise1}:
\begin{proposition} \label{noise2}
The recorded time $t_{ij}$ corresponding to time of shortest approach to the the $j^{th}$ sensor and the $i^{th}$ piece of the parabolic trajectory is unique and obeys:
$t_{ij}-t_{oi} \sim {Erlang}\left(2 \lambda_T, \frac{2 \lambda_T}{t_{ij}^*-t_{oi}}\right)$.
\end{proposition}

From (\ref{sol:sixb}) the estimation error in $\{s_i^*, \theta_i^* \}$ can be characterized as:

\bea \la{sol:sevena}
(\mathbb{X}_3^* - \mathbb{X}_3) =(\mathbb{A}_3^T \mathbb{A}_3)^{-1} \mathbb{A}_3^T  \underbrace{\begin{bmatrix}
      y_1(t_{i1}^*-t_{i1} )     \\[0.2em]
      y_2 (t_{i2}^*-t_{i2} )      \\[0.2em]
       \vdots       \\[0.2em]
      y_n( t_{in}^*-t_{ik} )     \\[0.2em]
     \end{bmatrix}}_{\Delta \mathbb{Y}_3}
\eea

From (\ref{sol:sevena}), the following lemma follows:

\begin{lemma} \label{lem:fivea}
Under Assumption \ref{noise2} and (\ref{sol:sevena})
$$\lim_{k\to\infty} (\mathbb{X}_3^* - \mathbb{X}_3)  = \Delta \mathbb{X}_3 \rightarrow [0 ~  0 ~  0 ~0 ~ 0 ~0]^T.$$
\end{lemma}
\begin{proof}
\bea \nonumber
\begin{split}
&\Delta \mathbb{X}_3 = (\mathbb{A}_3^T \mathbb{A}_3)^{-1}   \begin{bmatrix}
      \sum\limits_{k=1}^{n} y_k(t_{ik}^*-t_{ik} )     \\[0.2em]
      \sum\limits_{k=1}^{n} y_k(t_{ik}^*-t_{ik} ) (t_{ik}-t_{oi} )^3    \\[0.2em]
      \sum\limits_{k=1}^{n} y_k(t_{ik}^*-t_{ik} ) (t_{ik}-t_{oi} )^2    \\[0.2em]
      \sum\limits_{k=1}^{n} y_k(t_{ik}^*-t_{ik} ) (t_{ik}-t_{oi} )    \\[0.2em]
      \sum\limits_{k=1}^{n} y_k x_k (t_{ik}^*-t_{ik} )    \\[0.2em]
      \sum\limits_{k=1}^{n} y_k^2 (t_{ik}^*-t_{ik} )    \\[0.2em]
     \end{bmatrix} \\
     &\lim_{k\to\infty}  \Delta \mathbb{X}_3  \rightarrow    (\mathbb{A}_3^T \mathbb{A}_3)^{-1}   \begin{bmatrix}
      0 & 0 & 0 & 0 & 0 & 0   \\[0.2em]     \end{bmatrix} ^T.
      \end{split}
\eea
The results stems from the fact that $\{x_k, y_k, t_{ik}^* -t_{ik} \}, ~k \in \{1, \cdots, n\}$ are all zero mean random variables.
\end{proof}
\begin{corollary} \label{cor:fivea}
Under Assumption \ref{noise2} and (\ref{sol:sevena})
$\lim_{k\to\infty} [\gamma_i^* -\gamma_i, s_i^* -\hat{s}_i, \theta_i^*-\hat{\theta}_i , x_{oi}^*-x_{oi}, y_{oi}^*-y_{oi}] $$\rightarrow [0 ~  0 ~0 ~0 ~0~0 ].$
\end{corollary}
\section{Upper bound on achievable error variance of estimated parameters} \la{sab}
We will now characterize the upper bound on the Cram\'{e}r-Rao lower bounds on two of the parameters of the linear trajectory. We will not tackle the bounds on the parabolic trajectories because it is intractable to characterize the time of shortest approach in terms of the parameters of the parabola only. Hence the exact analytical expressions for the entries of the Fisher information matrix is not available. 

For the linear trajectories, observe also, that the estimation of the parameters $\{x_{oi}, y_{oi}\}$ depend on the difference between the transition times $\{t_{ij1}-t_{oi}, t_{ij2}-t_{oi}\}$. Again, we do not have analytic expressions for these transition times in terms of only the parameters of the line. We can also show that using the time of shortest distance of approach only to estimate the $\{x_{oi}, y_{oi}\}$ parameters, the resulting Fisher information matrix is rank deficient, confirming that the estimates of $\{x_{oi}, y_{oi} \}$ is unbounded if only the $t_{ij}$ data is used. The estimates of $\{s_i^*, \theta_i^*\}$ is however different. We will use $t_{ij}-t_{oi}$ to upper bound the Cram\'{e}r-Rao lower bound. This is because we are using a derivative of the actual recorded data, the $\{t_{ij1}-t_{oi}, t_{ij2}-t_{oi}\}$ data. The joint probability density function of the received $t_{ij}-t_{oi}$ recordings, or the mean of the $\{t_{ij1}-t_{oi}, t_{ij2}-t_{oi}\}$ times becomes: 

\scriptsize
\bea
\prod\limits_{j=1}^n \frac{\left(\frac{2 \lambda_T}{t_{ij}^*-t_{oi}}\right)^{2 \lambda_T} (t_{ij}-t_{oi})^{2 \lambda_T -1} e^{-\frac{2 \lambda_T (t_{ij}-t_{oi})}{t_{ij}^*-t_{oi}}}}{2 \lambda_T -1}
\eea
\normalsize

Consequently, the log likelihood function, denoted $\mathbb{L}_{(s_i, \theta_i, x_{oi}, y_{oi})}$ becomes:
\bea \la{ml}
\begin{split}
&n\log{\frac{(2 \lambda_T)^{2 \lambda_T}}{2 \lambda_T -1}} +(2 \lambda_T -1)  \sum\limits_{j=1}^n \log{(t_{ij} -t_{oi})}\\
& -2 \lambda_T \sum\limits_{j=1}^n \log{\frac{1}{s_i^*}((x_j-x_{oi}^*)\sin{\theta_i^*}+(y_j-y_{oi}^*)}\\
& -2 \lambda_T \sum\limits_{j=1}^n \frac{s_i^* (t_{ij} -t_{oi})}{((x_j-x_{oi}^*)\sin{\theta_i^*}+(y_j-y_{oi}^*)\cos{\theta_i^*}) } 
\end{split}
\eea

From (\ref{ml}), the entries of the Fisher information matrix are found as:
\bea \la{t1}
\begin{split}
-\mathbb{E} &\left( \frac{\partial^2 \mathbb{L} _{(s_i, \theta_i, x_{oi}, y_{oi})}}{\partial s_i^2}\right)= \frac{2 n \lambda_T}{s_i^{*2}}
\end{split}
\eea
\bea \la{t2}
\begin{split}
-&\mathbb{E} \left( \frac{\partial^2 \mathbb{L} _{(s_i, \theta_i, x_{oi}, y_{oi})}}{\partial s_i \partial \theta_i}\right)=\\ &-2 \frac{\lambda_T}{s_i^{*}} \sum\limits_{j=1}^n \frac{(x_j-x_{oi}^*)\cos{\theta_i^*}-(y_j-y_{oi}^*)\sin{\theta_i^*}}{(x_j-x_{oi}^*)\sin{\theta_i^*}+(y_j-y_{oi}^*)\cos{\theta_i^*}}
\end{split}
\eea
and 
\bea \la{t3}
\begin{split}
-&\mathbb{E} \left( \frac{\partial^2 \mathbb{L} _{(s_i, \theta_i, x_{oi}, y_{oi})}}{\partial \theta_i^2}\right)=\\ &2\lambda_T \sum\limits_{j=1}^n \frac{((x_j-x_{oi}^*)\cos{\theta_i^*}-(y_j-y_{oi}^*)\sin{\theta_i^*})^2}{((x_j-x_{oi}^*)\sin{\theta_i^*}+(y_j-y_{oi}^*)\cos{\theta_i^*})^2}
\end{split}
\eea
From (\ref{t1}) -- (\ref{t3}), the upper bound on the estimation errors in $s_i^*$ and $\theta_i^*$ are as follows:

{\scriptsize
\bea \nonumber
\underset{\operatorname{argmin}} {\|s_i^* -\hat{s}_i \|^2}\le \frac{s_i^{*2}}{2 D \lambda_T} \sum\limits_{j=1}^n \left(\frac{(x_j-x_{oi}^*)\cos{\theta_i^*}-(y_j-y_{oi}^*)\sin{\theta_i^*}}{(x_j-x_{oi}^*)\sin{\theta_i^*}+(y_j-y_{oi}^*)\cos{\theta_i^*}}\right)^2
\eea
}
and 
{\scriptsize
\bea 
\underset{\operatorname{argmin}} {\|\theta_i^* -\hat{\theta}_i \|^2}\le \frac{1}{2 D \lambda_T}
\eea
}
where 
{\scriptsize 
\bea 
\begin{split}
D =\sum\limits_{j=1}^n \sum\limits_{k \ne j}^n \frac{(x_j-x_{oi}^*)\cos{\theta_i^*}-(y_j-y_{oi}^*)\sin{\theta_i^*}}{((x_j-x_{oi}^*)\sin{\theta_i^*}+(y_j-y_{oi}^*)\cos{\theta_i^*})^2} \times \\
\frac{(x_j-x_{oi})(y_k-y_{oi})-(y_j-y_{oi})(x_k-x_{oi})}{(x_k-x_{oi}^*)\sin{\theta_i^*}+(y_k-y_{oi}^*)\cos{\theta_i^*}} 
\end{split}
\eea
}
\begin{remark}
Observe that the error bounds on the estimates are inverse law proportional to $\lambda_T$. In fact, the {\it signal-to-noise-ratio,(SNR),} is easily characterized in terms of $\lambda_T$. The expected value of the transition times are $t_{ijk}^*-t_{oi}, k \in \{1,2\}$ whiles the variances are $\frac{(t_{ijk}^{*}-t_{oi})^2}{\lambda_T}$. Using the so called conventional definition of SNR as the inverse coefficient of variation:
$$SNR = \sqrt{\lambda_T}.$$ We will use this characterization of the SNR to illustrate the performance of our algorithm in the simulations section.
\end{remark}

\section{Simulations}  \la{ssim}
We consider a comprehensive set of simulations for both linear trajectories and parabolic trajectories. First we consider the noise free case, where the exact time stamps are available. Fig \ref{fig:fig1} shows four linear trajectories being tracked with zero noise in the recorded time stamps. On each line, only three sensors are used to estimate the linear joins. The estimated line is the black line within the large colored line boundaries. 
\begin{figure}[!htb]
\centering
\includegraphics[scale=0.5]{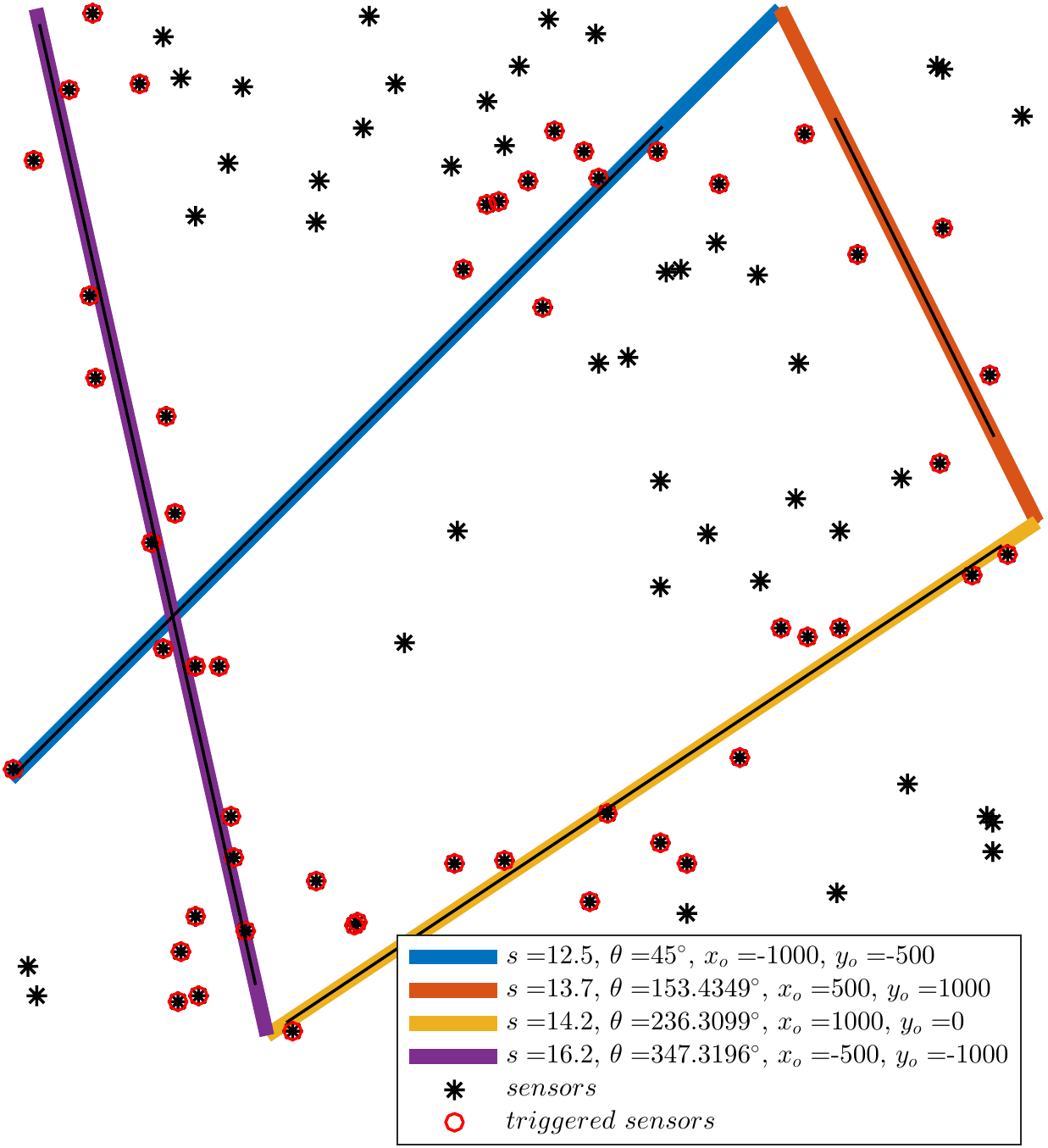}
\caption{ Piece-wise linear tracking with no noise in recorded time stamps}
  \label{fig:fig1}
\end{figure}
 Fig. \ref{fig:fig2} also shows the nonlinear counterpart of   Fig. \ref{fig:fig1}.  On each parabola, only six sensors are used to estimate the parabolic joins. The estimated parabola is the black curve within the boundaries of the large colored curves. 
\begin{figure}[!htb]
\centering
\includegraphics[scale=0.5]{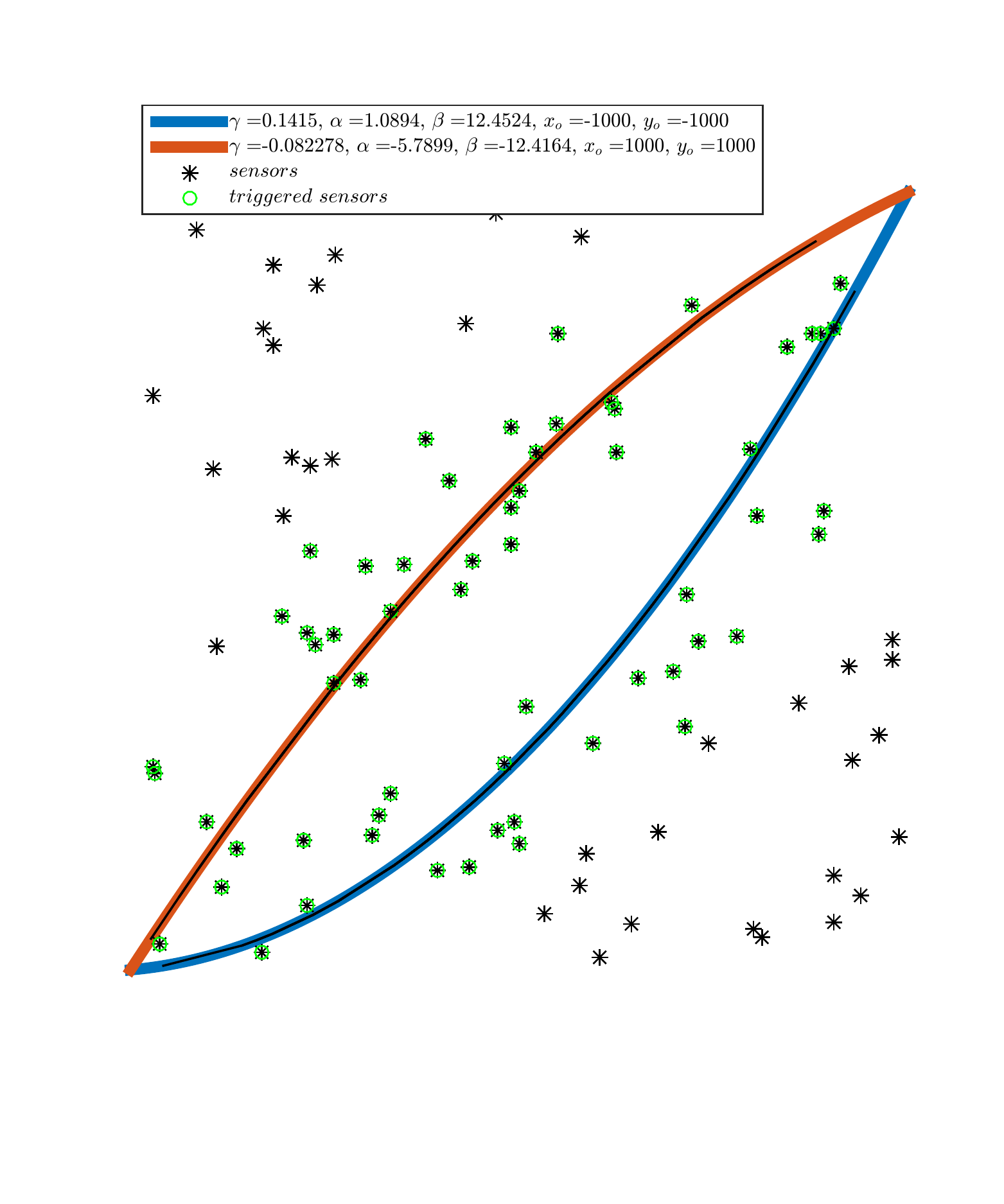}
\caption{ Piece-wise parabolic tracking with no noise in recorded time stamps}
  \label{fig:fig2}
\end{figure}
The two figures above illustrate that in the noise free case, our algorithm achieves perfect localization of the trajectories of the radiation sources. It is however clear that in practice, we would not have noise free time stamps. We turn our attention now to simulations with noise.

From  Fig. \ref{fig:fig3} --  Fig. \ref{fig:fig6}, we present results of $10^5$ Monte-Carlo simulations to illustrate the robustness of our algorithm to the presence of real noise. In all cases we present three different number of sensors on the same plot: $\{20,50,500\}$. In particular, we show the performance of our algorithm in the noisy case. The parameters $\{s, \theta\}$ are compared to our derived upper bound of the Cram\'{e}r-Rao lower bound. Here, all sensors are placed in a $2km$ by $2km$ square grid centered about zero. There are $\approx 1000$ sensors whose locations are uniformly distributed within the square grid. The true parameters being estimated are $\{s^*, \theta^*, x_o^*, y_o^*\} = \{30, 45^{\circ}, -1000,500\}$. In all cases, the minimum distance from which a sensor is triggered is 170m. The expected background noise is maintained at $\lambda_b =1$ photons per second. For all situations, $\alpha_s = 0.0068$. $\lambda_s$ is varied to keep the minimum distance for triggering of the sensors same for all SNR levels, which translates to the same expected transition time stamps for equitable comparison of performance. 
\begin{remark}
Notice that the SNR cannot go beyond 0 because the minimum allowable count to trigger the sensor is 1 photon. 
\end{remark}
In all simulations, it is seen that at an SNR of about 10dB, the performance for all the variables being estimated is good, with a percentage error well below or around 1\%. 
\begin{figure}[!htb]
\centering
\begin{subfigure}{.25\textwidth}
  \centering
\includegraphics[width=.9\linewidth]{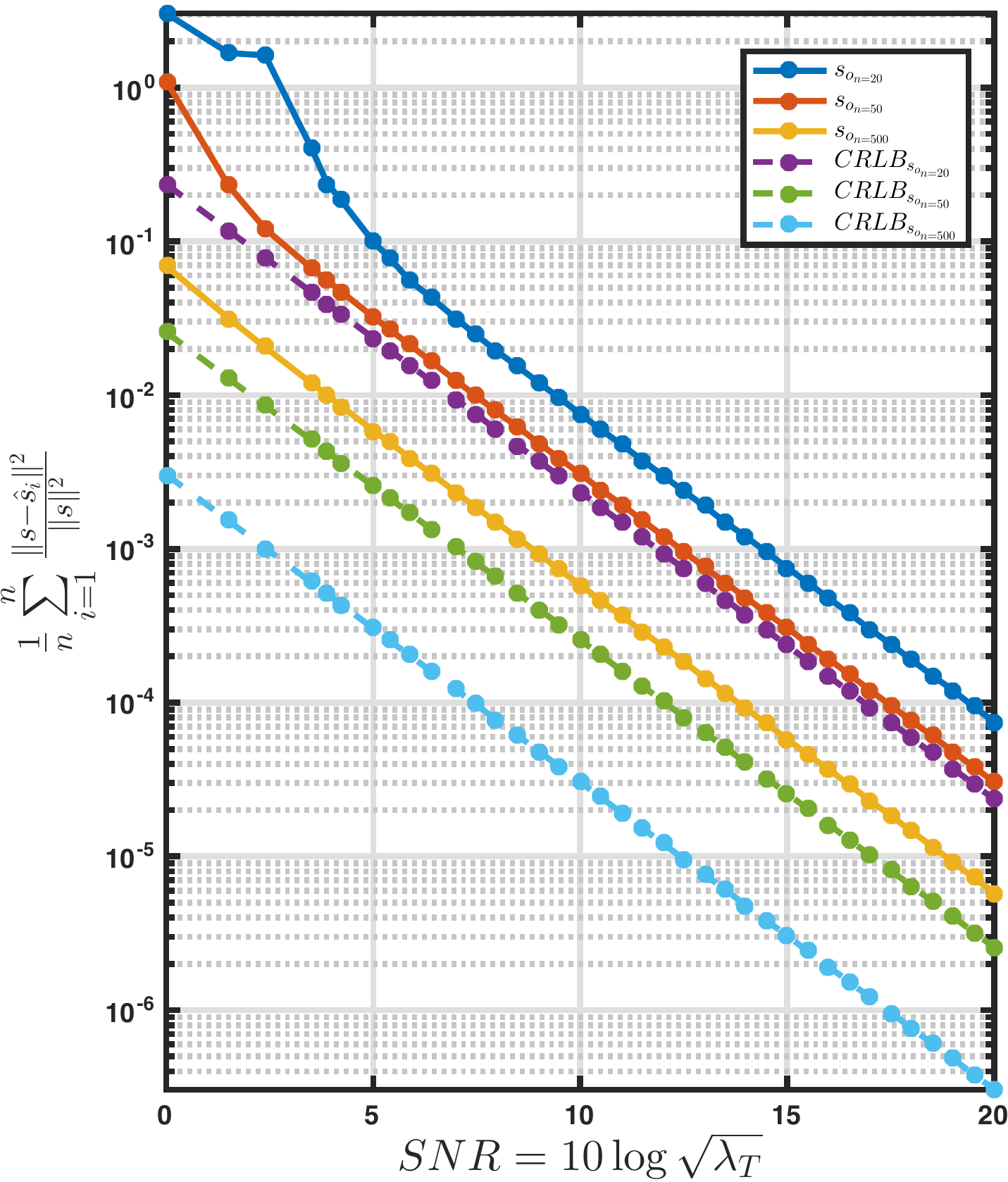}
\caption{ Estimation of $s$}
  \label{fig:fig3}
\end{subfigure}%
\begin{subfigure}{.25\textwidth}
  \centering
 \includegraphics[width=.9\linewidth]{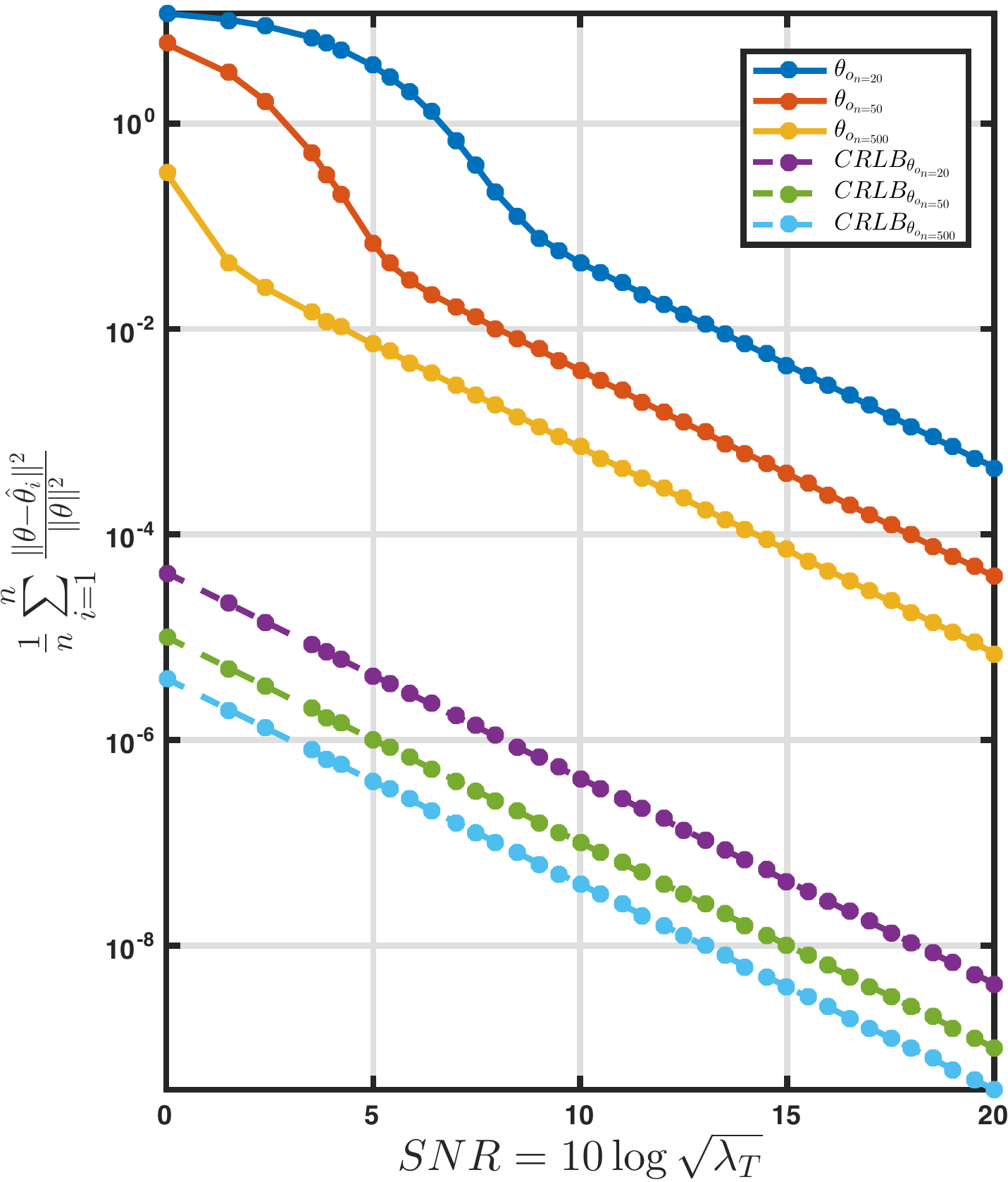}
\caption{ Estimation of $\theta$}
  \label{fig:fig4}
\end{subfigure}
\caption{Performance of $s$ and $\theta$ in the presence of noise for linear trajectories.}
\label{fig:test}
\end{figure}
\begin{figure}[!htb]
\centering
\begin{subfigure}{.25\textwidth}
  \centering
\includegraphics[width=.9\linewidth]{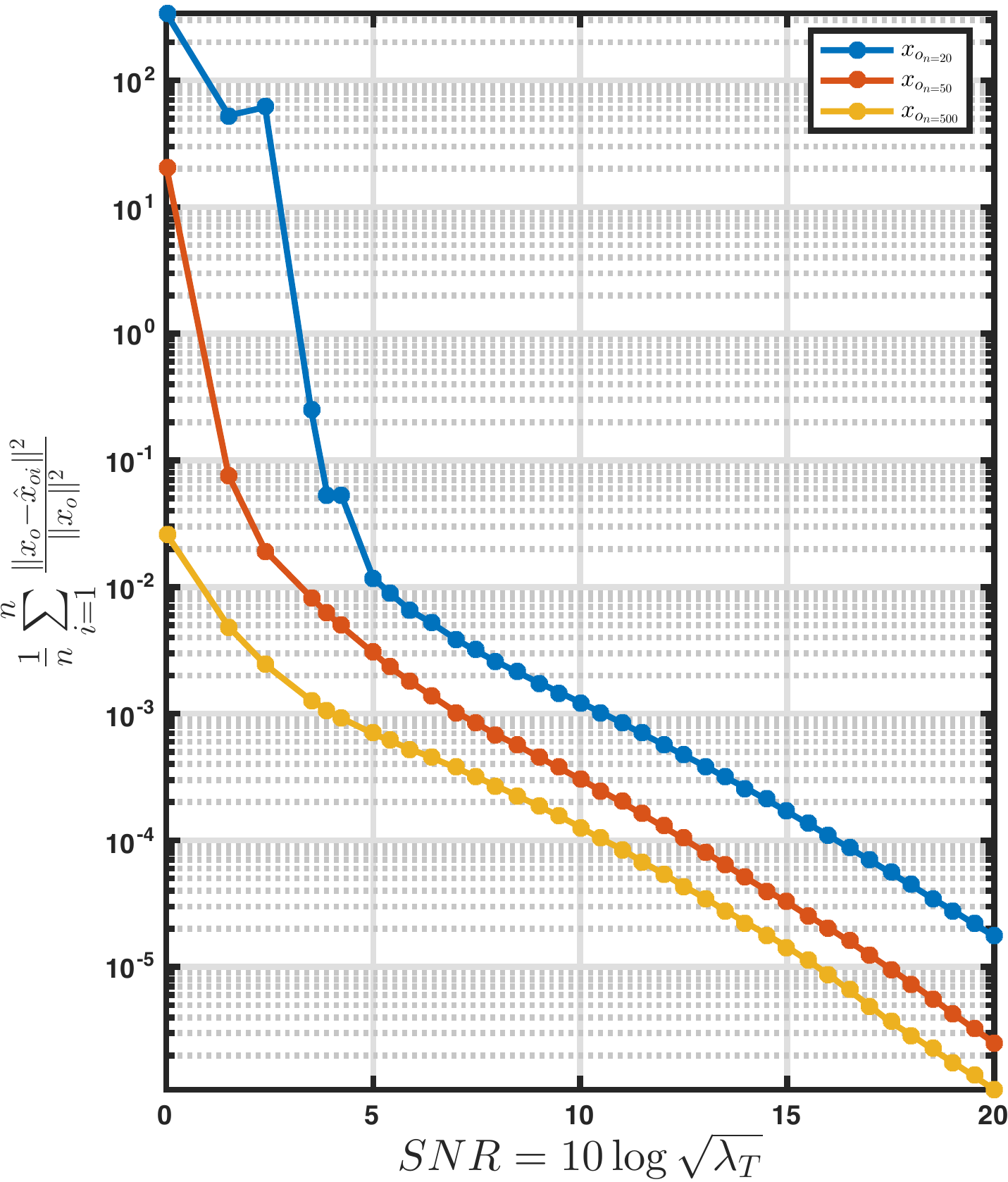}
\caption{ Estimation of $x_o$ }
  \label{fig:fig5}
  \end{subfigure}%
\begin{subfigure}{.25\textwidth}
  \centering
\includegraphics[width=.9\linewidth]{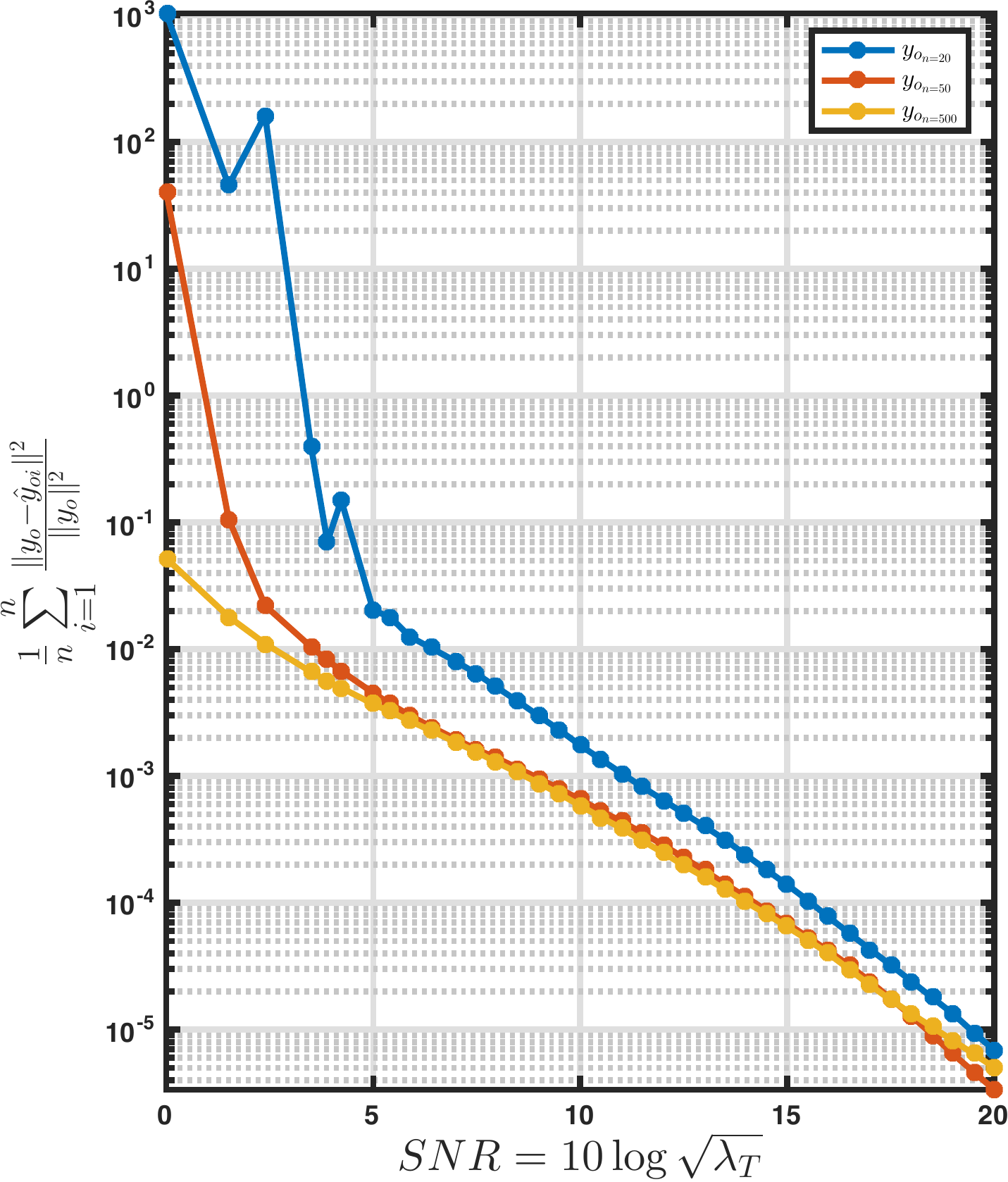}
\caption{ Estimation of $y_o$ }
  \label{fig:fig6}
  \end{subfigure}
\caption{Performance of $x_o$ and $y_o$ in the presence of noise for linear trajectories.}
\label{fig:test}
\end{figure}
From  Fig. \ref{fig:fig7} --  Fig. \ref{fig:fig11}, we present results of $10^5$ Monte-Carlo simulations on tracking of parabolic trajectories. Again,  we present three different number of sensors on the same plot: $\{20,50,200\}$. Again, all sensors are placed in a $2km$ by $2km$ square grid centered about zero. There are $\approx 500$ sensors whose locations are uniformly distributed within the square grid. The true parameters being estimated are $\{alpha^*, \beta^*, \gamma^*, x_o^*, y_o^*\} = \{29.89, 2.61, 0.82,  -1000,-1000\}$. In all cases, the minimum distance from which a sensor is triggered is 170 units. The expected background noise is maintained at $\lambda_b =1$ photons per second. For all situations, $\alpha_s = 0.0068$. $\lambda_s$ is varied to keep the minimum distance for triggering of the sensors same for all SNR levels, which translates to the same expected transition time stamps for equitable comparison of performance. 
In all simulations, it is seen that at an SNR of about 20dB, the performance for all the variables being estimated when using the number of sensors are $50$ or $200$ is good, with a percentage error well below or around 5--1\%. 
\begin{remark}
The simulation results for the parabolic trajectories gives an insight to realistic values of number of sensors or thresholds to use. We either end up using a very high number of sensors or a very high threshold. Note that very high thresholds mean all sensors will almost entirely be lined up along the trajectory of the source. This will make the observation matrix with the sensor locations column rank deficient and affect the inversion because of bad condition numbers. A high threshold on the other hand will mean weak signals will escape detection. A balance between number of sensors and threshold value will have to be made in practice.
\end{remark}
\begin{figure}[!htb]
\centering
\begin{subfigure}{.25\textwidth}
  \centering
\includegraphics[width=.9\linewidth]{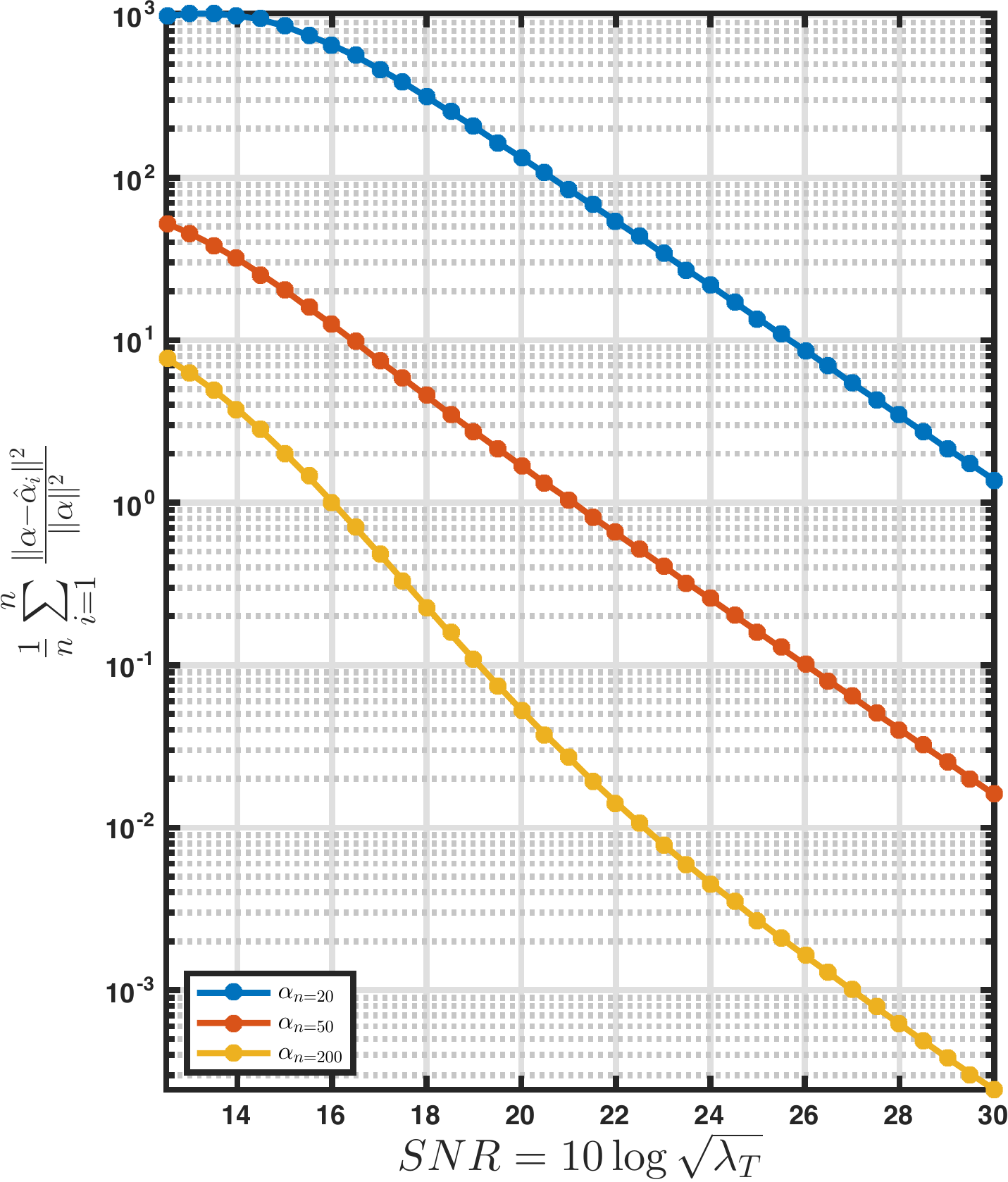}
\caption{ Estimation of $\alpha$ }
  \label{fig:fig7}
  \end{subfigure}%
\begin{subfigure}{.25\textwidth}
  \centering
\includegraphics[width=.9\linewidth]{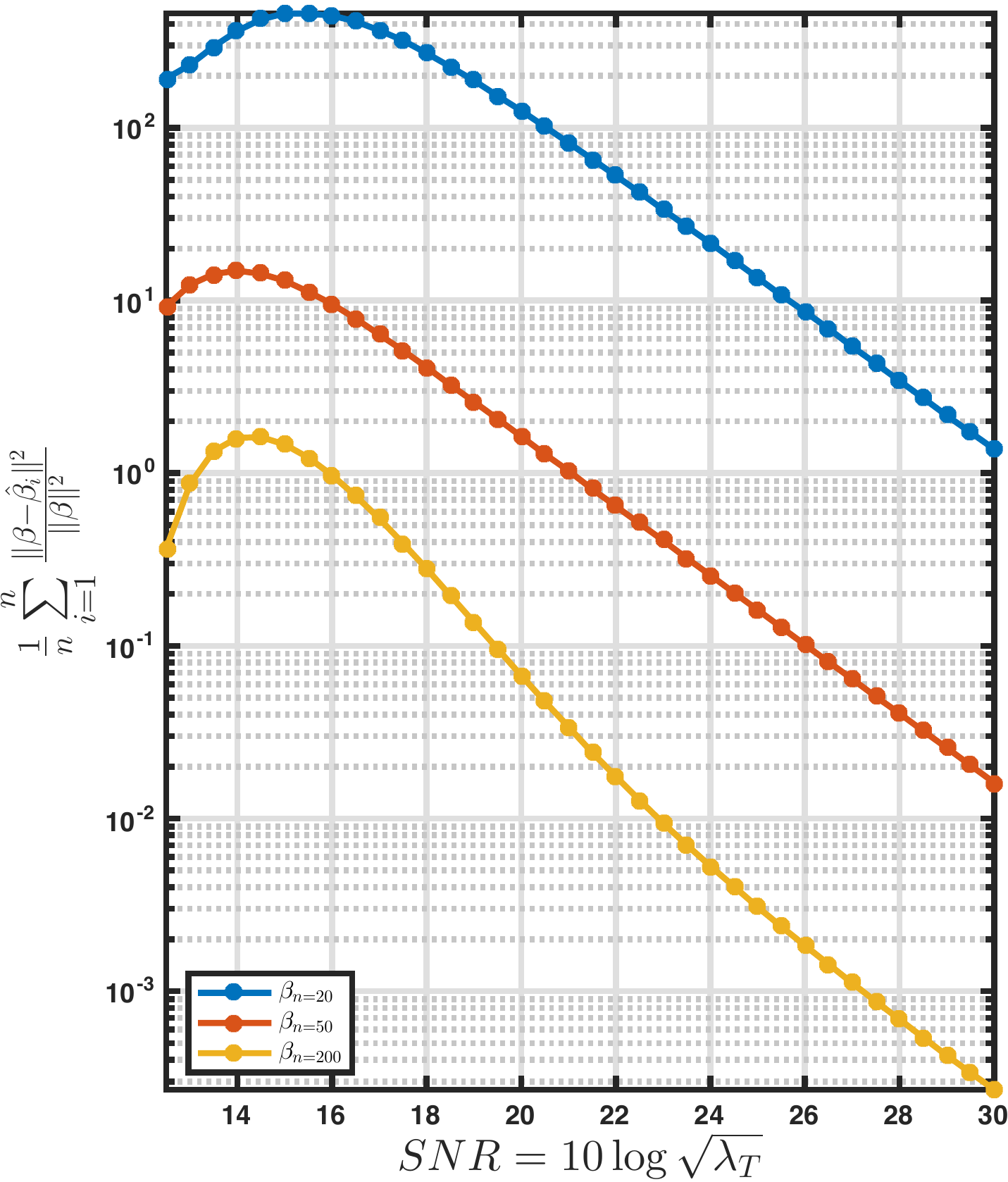}
\caption{ Estimation of $\beta$ }
  \label{fig:fig8}
  \end{subfigure}
\caption{Performance of $\alpha$ and $\beta$ in the presence of noise for parabolic trajectories.}
\label{fig:test}
\end{figure}

\begin{figure}[!htb]
\centering
\begin{subfigure}{.25\textwidth}
  \centering
\includegraphics[width=.9\linewidth]{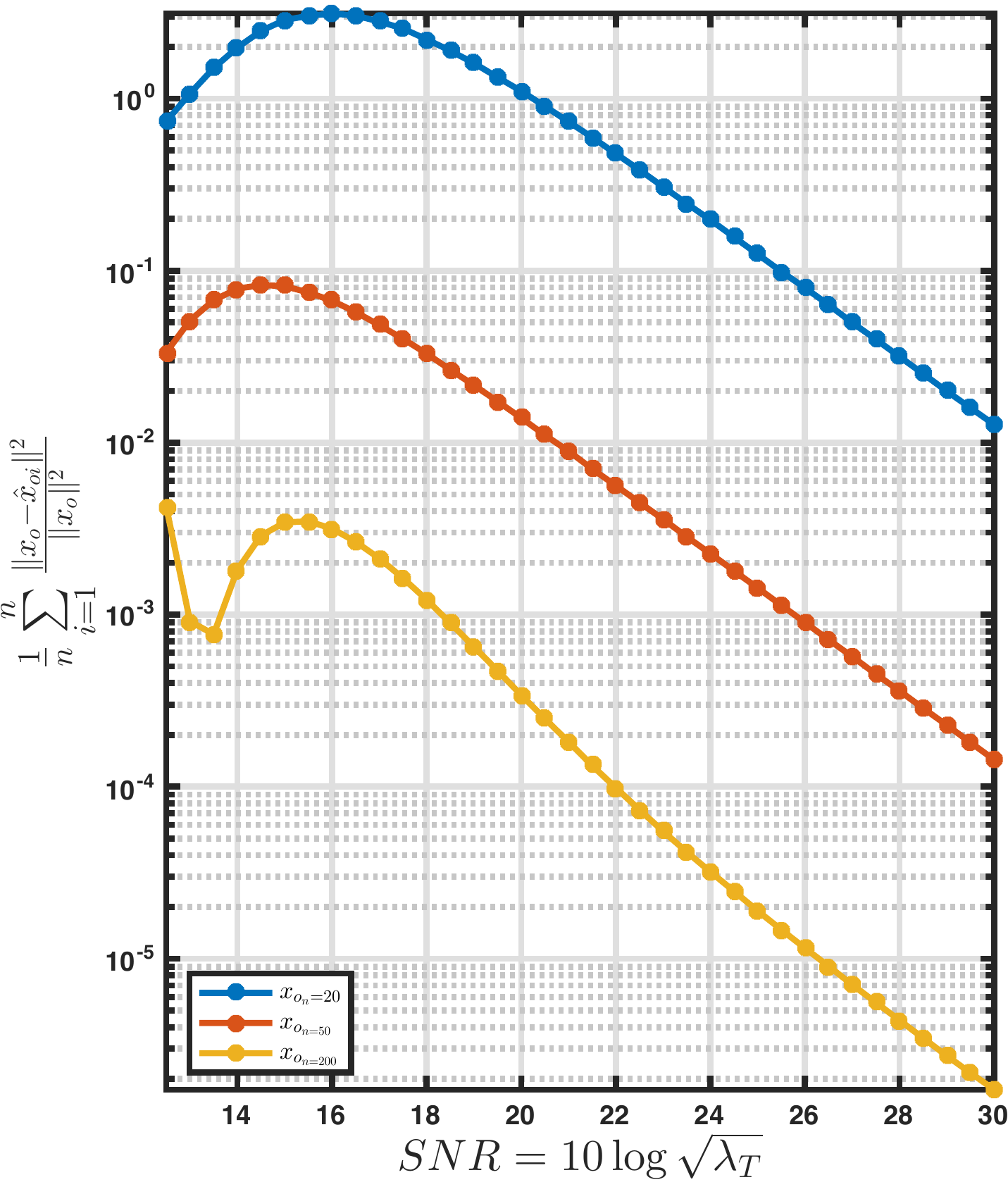}
\caption{ Estimation of $x_o$ }
  \label{fig:fig9}
  \end{subfigure}%
\begin{subfigure}{.25\textwidth}
  \centering
\includegraphics[width=.9\linewidth]{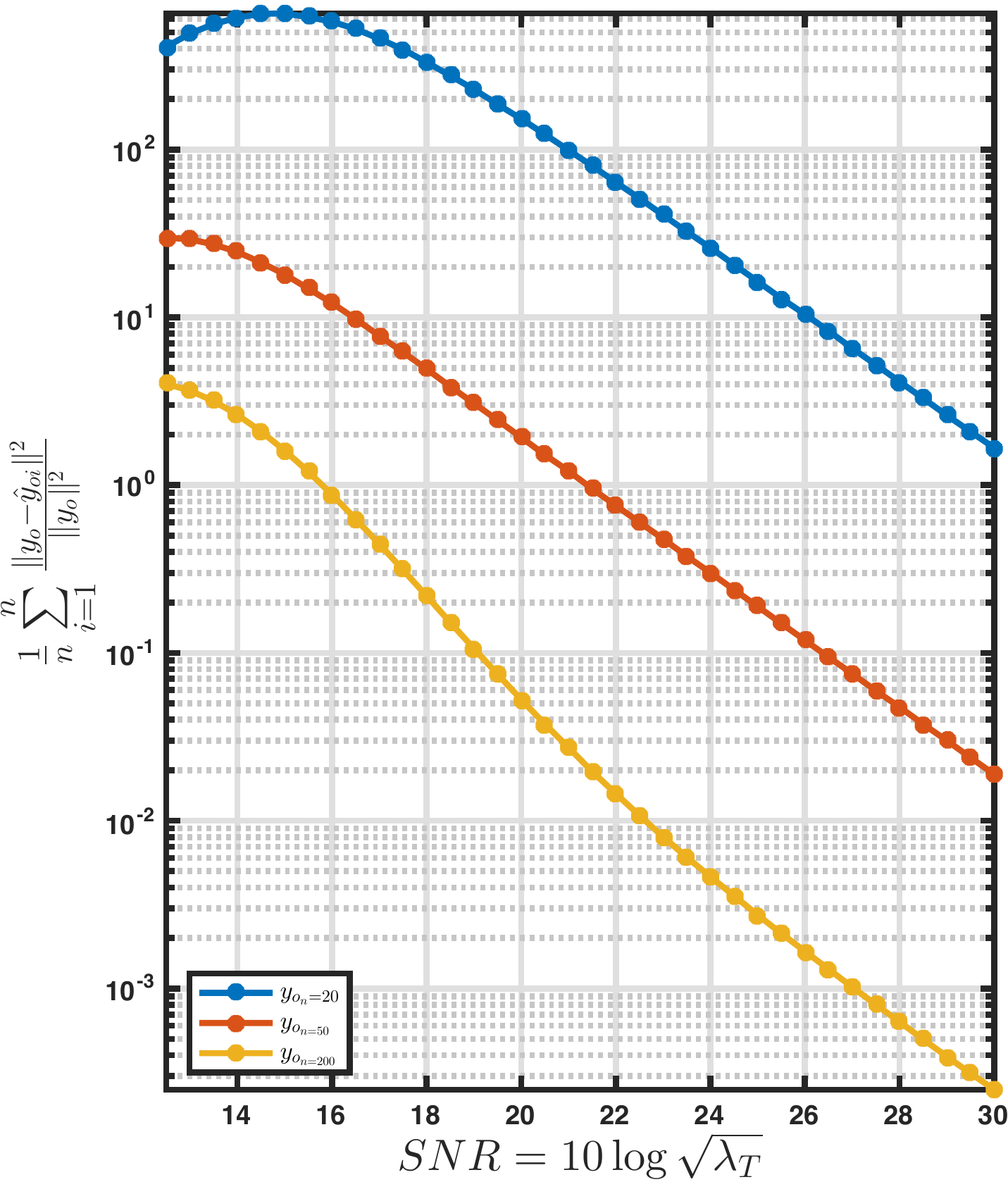}
\caption{ Estimation of $y_o$ }
  \label{fig:fig10}
  \end{subfigure}
\caption{Performance of $y_o$ and $x_o$ in the presence of noise for parabolic trajectories.}
\label{fig:test}
\end{figure}

\begin{figure}[!htb]
\centering
\begin{subfigure}{.25\textwidth}
  \centering
\includegraphics[width=.9\linewidth]{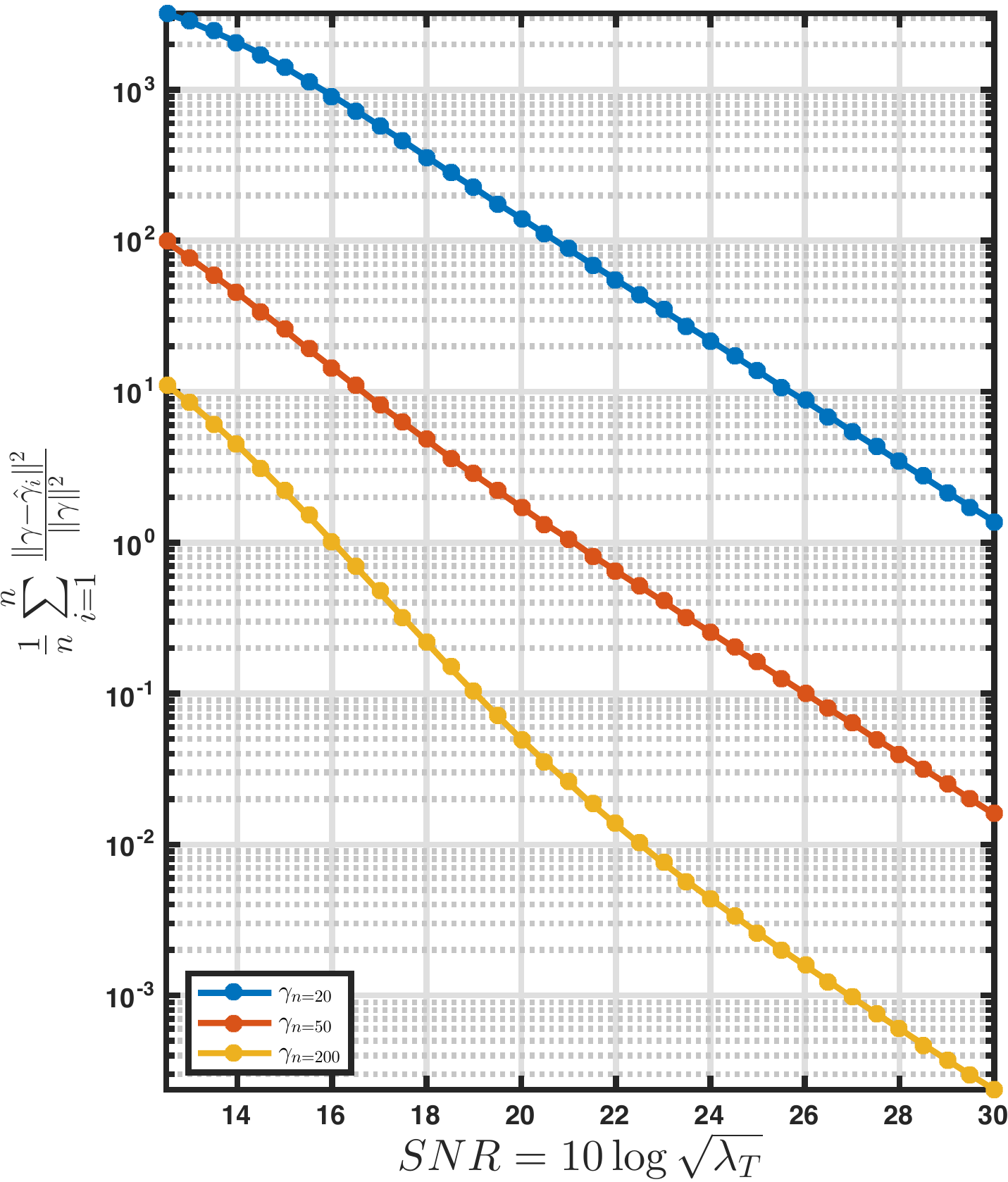}
\caption{ Estimation of $\gamma$ }
  \label{fig:fig11}
  \end{subfigure}%

\caption{Performance of $\gamma$ in the presence of noise for parabolic trajectories.}
\label{fig:test}
\end{figure}

\section{Conclusion}\la{scon}

We have considered robust and scalable algorithms for tracking of piece-wise linear and parabolic trajectories using binary proximity sensors. We have shown that, with piece-wise linear trajectories, three generically placed sensors suffice to track a single linear piece precisely in the noise free case. We have presented analytical and simulated results to show the robustness of our least squares algorithm. We have shown that, even with unknown statistics of the source, we require far fewer least squares problems, exactly $2^{n-1}-2$ less least squares problems than the previous results in \cite{bai}. We have also shown that the number of sensors is one less than the number used in \cite{bai} and \cite{bw} for exact  tracking in the noise free case. Further, we have derived a theoretical upper bound on the achievable error variance for the constant speed and elevation angle of the  linear trajectory. We have also extended our work to parabolic trajectories, which mimic better trajectories of automobiles on highways. We have shown that with parabolic trajectories no more than six generic sensors are required in the noise free case. Our solution for parabolic trajectories however does not utilize the time difference of between entering and leaving a sensors sensing range. The total number of sensors may or may not reduce when that information is incorporated.

\bibliographystyle{abbrv}

\bibliography{refs}

\end{document}